\documentclass[12pt,a4paper]{article}
\usepackage[T1]{fontenc}
\usepackage{graphicx}
\usepackage{amssymb}
\usepackage{amsmath}
\usepackage{amsfonts}
\usepackage{amsthm}
\usepackage{graphicx} 
\usepackage{color}
\usepackage{tikz}
\usetikzlibrary{positioning}
\usepackage{latexsym}
\usepackage{amssymb}
\usepackage[all,cmtip]{xy}
\usepackage{mathrsfs}
\usepackage{rotating}
\usepackage{amsgen, amstext,amsbsy,amsopn, amsthm, amsfonts,amssymb,amscd,amsmath,euscript,enumerate,verbatim,calc}
\usepackage[dvips]{epsfig}
\usepackage{eepic}
\usepackage{fix-cm}
\usepackage[left=1in,right=1in,top=1in,bottom=1in]{geometry}

\theoremstyle{plain}
\newtheorem{thm}{Theorem}

\newtheorem{lem}{Lemma}

\newtheorem{ack}{Aknowledgements}
\theoremstyle{definition}
\newtheorem{defn}{Definition}
\theoremstyle{remark}

\numberwithin{equation}{section}

\providecommand{\keywords}[1] { \small \textbf{\textit{Keywords---}} #1 }

\title{An Equivalent form of Twin Prime Conjecture \\ connected with a sequence of arithmetic progressions}
\date{}
\author{Srikanth Cherukupally
\thanks{Email address: {\it sricheru1214@gmail.com}}}
\begin{document}
	\maketitle
\begin{abstract}
\noindent We give an equivalent form of the Twin prime conjecture 
relating to a symmetric property that is observed for terms present in a certain sequence of arithmetic progressions defined for a pair of co-prime integers.    
\end{abstract} 
\section{Introduction}
A given pair of numbers ($X$,$X+2$) is called a {\it Twin prime pair} if both $X$ and 
$X+2$ are prime numbers. The Twin prime conjecture states that there are infinitely many twin prime pairs. The conjecture is a long standing open problem in number theory. A landmark result \cite{zhang2013} by Yitang Zang proves infinitude of primes pairs ($P$, $P+n$) with a finite bounded gap $n$ between them. It is considered to be a major breakthrough towards proving the conjecture though the proven gap $n$ is a large number less than $70\times 10^6$. Immediately after the result, Polymath project \cite{poymath2014} initiated by Terrence Tao refined the gap $n$ to $4680$. An independent result \cite{maynard2019} by James Maynard reduced the bound to $600$. The expository article \cite{articletwin} gives a beautiful account of landmark results and dramatic events that were registered in the history of twin prime conjecture.  
 
In this paper, we observe a mirror image symmetry property (called {\it symmetricity}) satisfied by leading terms of arithmetic progressions, which together form a definite collection with respect to a modular inverse property. We call this collection of arithmetic progressions a {\it sequence} since the progressions (in the collection) are strictly ordered. The sequence (of arithmetic progressions) in our current context is defined for any pair of co-prime integers
($a_0$,$d_0$). We mainly prove that, for a fixed $d_0$, the sequence of arithmetic progressions defined for ($a_0$,$d_0$), with $1\leq a_0 < d_0$, has leading terms satisfy symmetricity property only when $a_0$ corresponds to a divisor of $d_0^2-1$. This correspondence result help us give an equivalent form of twin prime conjecture. Our arguments are elementary.  

The object, sequence of arithmetic progressions, is introduced for the first time in doctoral thesis \cite{sri_phd}, and is shown to be connected with the Euclid's algorithm of finding GCD \cite{sri_euc}, and some cryptographic applications and a new key sharing method in the context of public key cryptosystem are proposed in \cite{sri_keysharing, sri_phd}.  The notation and arguments that we use here in proving the claim are taken from a published work \cite{sri_euc}. However, the paper is self-contained.     

\section{Sequence of arithmetic progressions}
We follow the notation: $A(x,y)$ is an arithmetic progression with leading term $x$ and common difference $y$.

Let $A(a,d)$
and $A(a',d')$ be two arithmetic progressions such that their terms are connected by modular multiplicative property:
$$(a+id)(a'+id') \equiv 1 \pmod{a+(i+1)d}.$$ 
We call this property as {\bf Property $\mathcal{P}$}. 
We can see that the terms of the two progressions satisfy this property only when both $(a,d)$ and $(a'd')$ are co-prime pairs. Further, for a given $A(a,d)$, $A(a',d')$ is unique when $d'\leq d$. The uniqueness of existence of $A(a',d')$ is proved in the following result.   

\begin{lem}
	For a given $A(a,d)$ with $a$ and $d$ being co-prime, there exists only one progression $A(a',d')$ with $1\leq d' \leq d$ such that the terms of the two progressions are connected by Property $\mathcal{P}$.  
\end{lem}
\begin{proof}
	By Property $\mathcal{P}$, 
	$$(-d)(a' +id') \equiv 1 \pmod {a+(i+1)d)}.$$
	which implies
	$z_i = \frac{(a' +id')d+1}{a +(i+1)d}$ is an integer for any $i \geq 1$.
	Equivalently, 
	$z_{i+1} - z_i = \frac{Cd}{(a +id)(a +(i+1)d)}$ is an integer,
	where $C = a'd+dd'- ad' -1$. Since $a$, $d$, $a+id$ and $a+(i+1)d$ are pair-wise
	co-prime, $C$ is divisible by $a+id$ for any $i \geq 1$. 
	This implies that $C = 0$. The
	expression for C can be rewritten as
	\begin{equation}\label{eq_prop}
		a' = d' + \frac{ad'-1}{d}.
	\end{equation}

	In the above relation, $a'$ is an integer only if $d'$ is the multiplicative inverse of $a \pmod d$. There
	exists only one value for $d'$ that is less than d. When $d = 1$, we have 
	$d' = 1$ and $a' = a$. This proves the uniqueness of $A(a',d')$. 
\end{proof}
With the above result, one can inductively construct, from a given co-prime pair $(a_0,d_0)$, a unique sequence of arithmetic progressions whose common differences are in non-increasing order. The sequence is formally defined as follows for a given co-prime pair $(a_0,d_0)$. 
\begin{defn}
 $\mathfrak{S}(a_0,d_0)$  is the sequence consisting of distinct arithmetic progressions
	$$\left<A(a_0,d_0), A(a_1,d_1), A(a_2,d_2) \ldots \right>,$$ 
\end{defn}
\noindent where the terms of any two consecutive arithmetic progressions have Property $\mathcal{P}$ and common differences satisfy $d_i \geq d_{i+1} \geq 1$.

By equation \ref{eq_prop}, leading terms $a_i$ and $d_i$ satisfy the property:
\begin{eqnarray}\label{a_d_prop}
	d_{i+1} & \equiv & a_i\pmod{d_i} \nonumber \\
	a_{i+1} & = & d_{i+1} + \frac{a_id_{i+1}-1}{d_i}
\end{eqnarray}
\subsection*{Groupings of $\mathfrak{S}(a_0,d_0)$}
 A sub-collection $\mathcal{G}$ of consecutive progressions of 
 $\mathfrak{S}(a_0,d_0)$ is called a {\it grouping} if it
 satisfies the following two properties.
 \begin{itemize}
 	
 \item The difference between the common differences of any two consecutive progressions in $\mathcal{G}$ is same.
 \item $\mathcal{G}$ is maximal.
\end{itemize}
 We refer to the difference between consecutive common differences as the second
 common difference corresponding to $\mathcal{G}$. The size of $\mathcal{G}$ is the number of progressions
 in it, and is denoted by $|\mathcal{G}|$. 
 Note that any two consecutive groupings share an
 arithmetic progression.
 
 For example, the sequence $\mathfrak{S}(11,25)$ is as follows.
 \begin{eqnarray*}
 	& & 11, 36, 61, 86, ...\\
 	& &23, 39, 55, 71, ... \\
 	& & 17, 24, 31, 38, ... \\
 	& & 17, 22, 27, 32, ... \\
 	& & 13, 16, 19, 22, ... \\
 	& & 5, \, 6, 7, 8, ... \\
 	& & 5, 6, 7, 8, ... 
 \end{eqnarray*}
The sequence $\mathfrak{S}(11,25)$ has two groupings:
\begin{eqnarray*}
\mathcal{G}_1 & = & \left<A(11,25),A(23,16),A(17,7)\right>, \\
\mathcal{G}_2 & = & \left<A(17,7),A(17,5),A(13,3),A(5,1)\right>.
\end{eqnarray*}

The second common difference corresponding to $\mathcal{G}_1$ is $9$. The second common difference corresponding to $\mathcal{G}_2$ is $2$. The groupings share the progression $A(17,7)$. The
sizes of $\mathcal{G}_1$ and $\mathcal{G}_2$ are $3$ and $4$, respectively. 

\subsection*{Properties of terms within a grouping}
Let $\mathcal{G}$ be a grouping of $\mathfrak{S}(a_\alpha,d_\alpha)$ consisting of progressions
$$A(a_\alpha,d_\alpha),A(a_{\alpha+1},d_{\alpha+1}),...,A(a_{\beta} ,d_\beta),$$
for some $0\leq \alpha < \beta$. Let $\triangle$ be the second common difference corresponding to $\mathcal{G}$.
By the definition of grouping, $\triangle = d_r - d_{r+1}$, 
$\alpha \leq r \leq \beta-1$. 

\begin{lem}
For $\alpha \leq i \leq \beta$, 
$\lfloor \frac{a_i}{d_i} \rfloor = \lfloor \frac{a_\alpha}{d_\alpha} \rfloor + (i-\alpha)$ 
\end{lem}
\begin{proof}
		Firstly, note that, for $\alpha \leq i \leq \beta-1$,  
	\begin{eqnarray*}
		\frac{a_{i+1}\triangle+1}{d_{i+1}} - \frac{a_i\triangle+1}{d_i} 
		& = &  \triangle \bigg(\frac{a_{i+1}d_i - a_id_{i+1} +1}{d_id_{i+1}}\bigg) \\
		& = & \triangle \,\,\, \textrm{(By Equation \ref{a_d_prop})}
	\end{eqnarray*}
\end{proof}
By the above property, for a progression $A(a,d)$ in $\mathfrak{S}(a_0,d_0)$, 
$\lfloor \frac{a}{d} \rfloor$ gives the position index of the progression within the sequence $\mathfrak{S}$. The following result gives a defining equation satisfied by leading terms of the progressions in $\mathcal{G}$. 
\begin{lem}
	For $\alpha \leq i\leq \beta$, 
	$$a_i = a_\alpha + \triangle(\beta-i)(i-\alpha)(d_\beta -z_\alpha),$$  
	where $z_\alpha = \frac{a_\alpha\triangle+1}{d_\alpha}.$
\end{lem}
\begin{proof}

Further, for $\alpha \leq i \leq \beta-1$, 
\begin{eqnarray*}
	a_{i+1} - a_i & = & -\frac{a_{i+1}\triangle+1}{d_{i+1}} + \triangle + d_{i+1} \\
	& = & -\frac{a_{i+1}\triangle+1}{d_{i+1}} + \triangle + (d_\beta+(\beta-i-1)\triangle) \\
	& = & -(z_\alpha + (i+1-\alpha)\triangle) + d_\beta + (\beta-i)\triangle \\
	& = & (\alpha+\beta-1-2i)\triangle + d_\beta - z_\alpha. 
\end{eqnarray*}  
We thus have, for $\alpha \leq i \leq \beta$, 
\begin{eqnarray*}
	a_i - a_\alpha & = & \sum_{j=\alpha}^{i-1} (a_{j+1}-a_j) \\
	& = & \sum_{j=\alpha}^{i-1} \big(\triangle(\alpha+\beta-1-2j) + d_\beta-z_\alpha \big) \\
	& = & \triangle(\beta - i)(i-\alpha) + (i-\alpha)(d_\beta-z_\alpha). 
\end{eqnarray*}
\end{proof}

\subsection*{Symmetricity property of leading terms}
We can notice that leading terms of the progressions (in $\mathcal{G}$) are evaluations of a polynomial $f(x) = ax^2 +bx+c$ at integer values, where $a, b, c$ are some constants specific to $\mathcal{G}$. Since $a < 0$, $f(x)$ defines an inverted parabola. This gives an indication that some leading terms may appear more than once. This repetitive pattern of leading terms essentially gives rise to a mirror image symmetry about middle progression in the grouping. We call it {\it symmetricity} property. 

For example, the leading terms of first six initial progressions of 
$\mathfrak{S}(17,23)$ follow symmetricity. These 6 progressions form the first grouping. 
\begin{eqnarray*}
	& & 17, 40, 63, 86, 109 ... \\
	& & 33, 52, 71, 90, 109 ... \\
	& & 41, 56, 71, 86, 101 ... \\
	& & 41, 52, 63, 74, 85 ... \\
	& & 33, 40, 47, 54, 61 ... \\
	& & 17, 20, 23, 26, 29 ... \\
	& & 13, 15, 17, 19, 21 ... \\
	& & 7, 8, 9, 10, 11 ...
\end{eqnarray*}

Leading terms satisfy the property subjected to the quantity 
$d_\beta - z_\alpha$ in the defining equation of leading terms. 
In the following main result we prove a condition for symmetricity property in the first grouping $\mathcal{G}$.  
 
\begin{thm}
	The first grouping of $\mathfrak{S}(a_0,d_0)$ has symmetricity if and only if $\triangle$ divides $d_0^2-1$.  
\end{thm}
\begin{proof}
	Since $\mathcal{G}$ is the first grouping, $\alpha = 0$ and the leading terms' defining equation becomes 
		$$a_i = a_0 + \triangle(\beta-i)(i)+ i(d_\beta -z_0),$$ 
		where $z_0 = \frac{a_0\triangle+1}{d_0}$, and $\beta > 0.$. We have $z_0\equiv d_0^{-1} = d_\beta^{-1} \pmod \triangle$. If $\triangle$ divides
		$d_0^2-1$, then $d_\beta^2 \equiv 1 \pmod \triangle$, i.e., $d_\beta$ is self-invertible modulo $\triangle$. This implies 
		$z_0 = d_\beta$. So, $a_i = a_{\beta-i}$, $0\leq i\leq \beta$. 
		Conversely, if $a_i = a_{\beta-i}$, $0\leq i\leq \beta$, then 
		$z_0 = d_\beta$ and thus $\triangle$ divides $d_0^2-1$. 
\end{proof}
Let $S(d_0)$ be the number of positive integers $a_0 < d_0$ such that the leading terms of progressions in the first grouping of $\mathfrak{S}(a_0,d_0)$ exhibit symmetricity. There are total $\phi(d_0)$ values for $a_0 (<d_0)$ co-prime to $d_0$. Out of $\phi(d_0)$ numbers, there are $\frac{\sigma(d_0^2-1)}{2}$ values for $a_0$ that belong to the set $S(d_0)$. We can see that the size of $S(d_0)$ is 2 if and only if both $d_0-1$ and $d_0+1$ are primes. Thus, {\it proving the twin prime conjecture is equivalent to proving there exist infinitely many $d_0$ with $|S(d_0)| = 2$.}

\begin{ack}
	The author is thankful to Prof. B. Sury, (ISI, Bangalore), for his brief review of the manuscript and his endorsement for uploading it to Arxiv. 
\end{ack}
\end{document}